\documentclass{amsart}
\usepackage{amsthm, amssymb, amsfonts, amsmath, tikz}
\newtheorem{theorem}{Theorem}[section]
\newtheorem{lemma}[theorem]{Lemma}

\newtheorem{proposition}[theorem]{Proposition}
\theoremstyle{definition}

\theoremstyle{remark}

\newtheorem{problem}[]{Problem}

\newtheorem{observation}{Observation}

\title{Local Antimagic Coloring of Some Graphs}
\author{Ravindra Pawar}
\author{Tarkeshwar Singh}
\author{Adarsh Kumar Handa}
\author{Aloysius Godinho}
\thanks{Ravindra Pawar: \texttt{p20200020@goa.bits-pilani.ac.in}}
\thanks{Tarkeshwar Singh: \texttt{tksingh@goa.bits-pilani.ac.in}}
\thanks{Adarsh Kumar Handa: \texttt{adarsh@pccegoa.edu.in}}
\thanks{Aloysius Godinho: \texttt{aloysius314@gmail.com}}

\begin{document}

\maketitle
\begin{abstract}
Consider a simple graph $G$ without $K_2$ component with vertex set $V$ and edge set $E$. Local antimagic labeling $f$ of $G$ is a one-to-one mapping of edges to distinct positive integers $1,2, \dots, |E|$ such that the weights of adjacent vertices are distinct, where the weight of a vertex is the sum of labels assigned to the edges incident to it. These weights of the vertex induced by local antimagic labeling result in a proper vertex coloring of the graph $G$. We define the local antimagic chromatic number of $G$, denoted as $\chi_{la}(G)$, as the smallest number of distinct weights obtained across all possible local antimagic labelings of $G$. In this paper, we explore the local antimagic chromatic numbers of various classes of graphs, including the union of certain graph families, the corona product of graphs, and the necklace graph. Additionally, we provide constructions for infinitely many graphs for which $\chi_{la}(G)$ equals the chromatic number $\chi(G)$ of the graph.

{\setlength{\parindent}{0cm}\textbf{Keywords:} Antimagic Graph, Local Antimagic Graph, Local Antimagic Chromatic Number.}\\

{\setlength{\parindent}{0cm}\textbf{AMS Subject Classification 2021: 05C 78.}}
\end{abstract}

\section{Introduction}

The coloring problems in Graph Theory are one of the oldest, most widely known, and unsolved problems in mathematics. They have been the central research topic for centuries among graph theorists. Recently Arumugam et al.\cite{lac_arumugam}, and Bensmail et al.\cite{lac_bensmail} independently defined the notion of local antimagic labeling of a graph that induces proper vertex coloring. Arumugam et al.~\cite{lac_arumugam} studied the vertex coloring induced by local antimagic labeling. All the graphs considered throughout this paper are simple graphs without the $K_2$ component. For graph theoretic terminology and notations, we refer to West~\cite{west}.\\

Hartsfield and Ringel \cite{pearls_in_graph_theory} introduced the concept of  antimagic labeling of a graph. Given a graph $G = (V,E)$, let $f : E \rightarrow \{1, 2, \dots, |E|\}$ be a bijection. For each vertex $u \in V$, the weight of $u$ induced by $f$ is $w(u) = \displaystyle\sum_{uv \in E} f(uv)$. If the induced weights under $f$ of any two vertices of $G$ are distinct, then $f$ is called antimagic labeling of $G$, and the graph $G$, which admits such labeling, is called an antimagic graph.\\

An antimagic labeling $f$ of a graph $G$ is said to be \textit{local antimagic} if the weights induced by $f$ of adjacent vertices are distinct. Local antimagic labeling naturally induces a proper vertex coloring of a graph $G$. The \textit{local antimagic chromatic number} $\chi_{la}(G)$ of a graph $G$ is the minimum number of colors used over all colorings of $G$ induced by local antimagic labeling of $G$.\\

In \cite{lac_arumugam}, authors calculated the local antimagic chromatic number of a few families of graphs \textit{viz} path, cycle, wheel, etc. Furthermore, they conjectured that {\it every graph other than $K_2$ is local antimagic}. Haslegrave \cite{haslegrave_proof} proved this conjecture.\\

We will use magic rectangle and magic rectangle sets to obtain local antimagic labelings of some graphs. A \textit{magic rectangle} $MR(a,b)$ is an array whose entries are $\{1,2, \dots, ab\}$, each appearing once, with all its row sums equal to a constant $\rho = \frac{b(ab+1)}{2}$ and all its column sums equal to a constant $\sigma = \frac{a(ab+1)}{2}$. Froncek \cite{mrs1, mrs2} generalized the idea of magic squares to magic rectangle sets. A \textit{magic rectangle set} $\mathcal{M}  = MRS(a, b; c)$ is a collection of $c$ arrays $(a\times b)$ whose entries are elements of $\{1, 2, \dots, abc\}$, each appearing once, with all row sums in every rectangle equals to a constant $\rho = \frac{b(abc+1)}{2}$ and all column sums in every rectangle equal to a constant $\sigma = \frac{a(abc+1)}{2}$.\\

In this paper, we investigate the local antimagic chromatic numbers of the union of some families of graphs, the corona product of graphs, the necklace graph, and we construct infinitely many graphs satisfying $\chi_{la}(G) = \chi(G)$.

\section{Known Results}

The following century-old existence result was given by Harmuth \cite{mr1,mr2}, which gives the necessary and sufficient conditions for the existence of a magic rectangle of a given order.

\begin{theorem}\cite{mr1, mr2} \label{mr}
A magic rectangle $MR(a,b)$ exists if and only if $a,b > 1$, $ab>4$, and $a \equiv b(\bmod\ 2)$.
\end{theorem}
Froncek \cite{mrs1, mrs2} proved the existence of MRS$(a,b;c)$.  
\begin{theorem} \cite{mrs2} \label{mrs2.3}
Let $a, b, c$ be positive integers such that $1 < a \le b$. Then a magic
rectangle set MRS$(a, b; c)$ exists if and only if either $a, b, c$ are all odd, or $a$ and $b$ are both even, $c$ is arbitrary, and $(a,b) \ne (2,2)$.
\end{theorem}

\begin{theorem}\cite{lac_pendent_lau}\label{th:lac_bound}
Let $G$ be a graph having $k$ pendants. If $G$ is not $K_2$, then $\chi_{la}(G) \ge k + 1$ and the bound is sharp.
\end{theorem}
\begin{theorem}\cite{lac_union_baca}\label{baca1}
Let $G$ be a $4r$-regular graph, $r \ge 1$. Then for every positive integer $m, \chi_{la}(mG) \le \chi_{la}(G)$.
\end{theorem}
\begin{theorem}\cite{lac_union_baca}\label{baca2}
Let $G$ be a $(4r+2)$-regular graph, $r \ge 0$,  containing a $2$-factor consisting only of even cycles. Then for every positive integer $m, \chi_{la}(mG) \le \chi_{la}(G)$.
\end{theorem}

\section{Local Antimagic Labeling of Union of Graphs}

With the knowledge of local antimagic chromatic numbers of various classes of well-known graphs, the researchers started calculating local antimagic chromatic numbers for graphs obtained from known graphs \cite{lac_corona_arumugam, lac_join_lau, lac_lau, lac_corona_sugeng}. Handa et al. \cite{adarsht} started by investing the local antimagic chromatic number of the union of paths, cycles, and complete bipartite graphs.\\

Bača et al.~\cite{lac_union_baca} investigated independently the local antimagic chromatic number and upper bounds for the union of paths, the union of cycles, the union of trees, and other graphs and their proof techniques are different from the proof techniques given in this paper.\\

 The union of two graphs $G_1 = (V_1,\  E_1)$ and $G_2 = (V_2,\ E_2)$ is the graph $G = (V,\ E)$ with vertex set $V = V_1 \cup V_2$ and edge set  $E = E_1 \cup E_2$.\\

Note that, if $G_1, G_2, \dots , G_n$ are graphs such that $\chi(G_i)=\chi_{i}$ then $\chi(\bigcup_i G_i)= \mbox{max}\{ \chi_{i}: 1 \leq i \leq n\}$. We have the following observation for the local antimagic chromatic number.
\begin{observation} \label{obs: boundunion}
For the graphs $G_1, G_2, \dots , G_n$, $\chi_{la} (G_j) \le \chi_{la}(\bigcup_{1 \le i \le n} G_i)$ for each $j$.
\end{observation}

\begin{theorem}
The graph $rP_n$ is local antimagic with $3 \le \chi_{la}(rP_n)\leq 2r+2$.
\end{theorem}

\begin{proof}
Let $ V(rP_n) = \{v_i^j : 0\leq j\leq r-1, 0\leq i\leq n-1\}$ be the vertex set of $rP_n$, where for each $j$, $v_0^j, v_1^j,\dots,v_{n-1}^j$
is a path. The lower bound is obvious from the observation \ref{obs: boundunion}. For the upper bound, we consider the following two cases.\\

{\bf Case 1:} $n$ is even.\\

We define edge labeling $ f:E\to \{1,2,\dots,nr-r\}$ as follow:
$$f(v_i^jv_{i+1}^j) =\begin{cases}

(r-\frac{j}{2})(n-1)- \frac{i}{2} &\mbox{if } i,j\equiv 0(\bmod\ 2)\\

  \frac{j}{2}(n-1) +\frac{i+1}{2} & \mbox{if } i\equiv 1(\bmod\ 2), j\equiv 0(\bmod\ 2)\\
 
  \frac{(j-1)}{2})(n-1)+ \frac{i+1}{2}+ \frac{n-1}{2} &\mbox{if } i\equiv 0(\bmod\ 2),  j\equiv 1(\bmod\ 2)\\
 
 (r-\frac{j-1}{2})(n-1)- \frac{i-1}{2}
 - \frac{n}{2} &\mbox{if } i,j\equiv 1(\bmod\ 2).
 \end{cases}$$ 
 
Then the induced vertex weights are  as follows:

$i\neq 0$ and $i \ne n-1$,
$$ w(v_i^j) =\begin{cases}

r(n-1) &  \mbox{if } i\equiv j(\bmod\ 2) \\
r(n-1)+1 & \mbox{if } i\not\equiv j(\bmod\ 2).
\end{cases}$$

For $i= 0$
$$ w(v_0^j) =\begin{cases}
(r-\frac{j}{2})(n-1) & \mbox{if } j\equiv 0(\bmod\ 2)\\

(n-1)(\frac{j-1}{2})+\frac{n}{2} & \mbox{if } j\equiv 1(\bmod\ 2).
\end{cases}$$
  
  For $i=n-1$
$$ w(v_{n-1}^j) =\begin{cases}

\frac{r - \frac{j}{2}}{2}(n-1)-\frac{n-2}{2} & \mbox{if } j\equiv 0(\bmod\ 2)\\

(n-1)(\frac{j+1}{2}) & \mbox{if } j\equiv 1(\bmod\ 2).
\end{cases}$$
  
  {\bf Case 2:} $n$ is odd.\\
  
We define edge labeling $ f:E\to \{1,2,\dots,nr-r\}$ as follow:
$$f(v_i^jv_{i+1}^j) =\begin{cases}
r(n-1)- \frac{j(n-1)}{2}-\frac{i}{2} &\mbox{if } i\equiv 0(\bmod\ 2)\\
\frac{j(n-1)}{2} +\frac{i+1}{2} & \mbox{if } i\equiv 1(\bmod\ 2).
\end{cases}$$ 

The induced vertex weights are  as follows:
$i\neq 0$ and $i \ne n-1$,
$$ w(v_i^j) =\begin{cases}

r(n-1) &  \mbox{if } i\equiv 0(\bmod\ 2) \\
r(n-1)+1 & \mbox{if } i\equiv 1(\bmod\ 2).
\end{cases}$$

$$ w(v_i^j) =\begin{cases}
r(n-1) - \frac{j(n-1)}{2}& \mbox{if }  i = 0 \\

(\frac{n-1}{2})(j+1) & \mbox{if } i= n-1.
\end{cases}$$
  
Since we have $2r+2$ distinct vertex weights, we conclude that $\chi_{la}(rP_n)\leq 2r+2$.
\end{proof}

Next, we investigate the local antimagic chromatic number for the union of cycles. Let the vertex set of $rC_n$ be
$ V(rC_n) = \{v_i^j : 1\leq j\leq r,~~  0\leq i\leq n-1\}$
where for each $j$, $v_0^j, v_1^j,\dots,v_{n-1}^j$
is a cycle of length $n$.
 
\begin{lemma}{\label{cycle1}}
If $n$ is even then the graph $rC_n$ is local antimagic with $\chi_{la}(rC_n)=3$.
\end{lemma}
\begin{proof}
By Observation \ref{obs: boundunion}, $\chi_{la}(C_n) = 3 \le \chi_{la}(rC_n)$. So it is sufficient to give a local antimagic labeling that induces exactly $3$ distinct weights. Consider the following edge labeling $f$ as 

\[f(v_i^jv_{i+1}^j) =   \begin{cases}
                        \frac{(j-1)n}{2}+ \frac{i}{2} +1 & \mbox{if } i\equiv 0(\bmod\ 2)\\

                        rn-\frac{(j-1)n}{2}- \frac{i-1}{2} & \mbox{if } i\equiv 1(\bmod\ 2)
                        \end{cases}\] 
Then the induced vertex weights are  as follows
\[ w(v_i^j) =   \begin{cases}
                rn+1 &  \mbox{if } i\equiv 1(\bmod\ 2), i\neq 0\\
                rn+2 & \mbox{if } i\equiv 0(\bmod\ 2), i\neq 0\\
                rn+\frac{4-n}{2} & \mbox{if } i=0
                \end{cases}\]
Hence $f$ is local antimagic labeling, and it induces $3$ weights as required.
\end{proof}
The illustration of local antimagic labeling of $2C_6$ is shown in Figure \ref{2C6}.
  \begin{figure}[ht]
        \centering
        \includegraphics[scale=0.6]{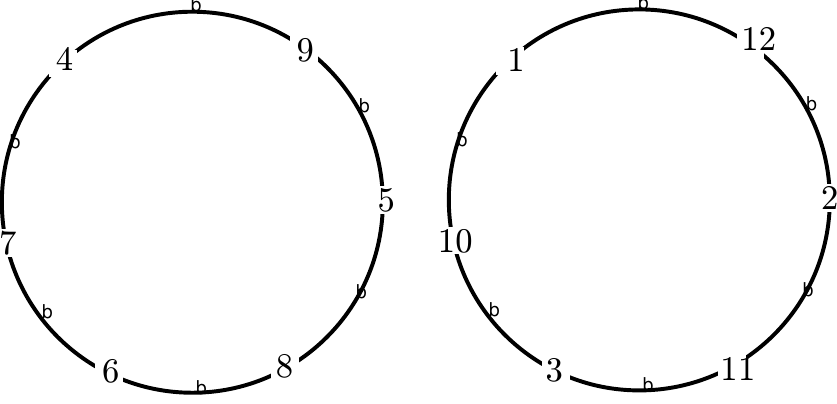}
        \caption{A local antimagic labeling of $2C_6$ with $\chi_{la}(2C_6) = 3$.}
        \label{2C6}
    \end{figure}

We give an upper bound on the local antimagic chromatic number for the union of odd-length cycles.
\begin{lemma}{\label{cycle2}}
    If $n$ is odd then the graph $rC_n$ is local antimagic with $\chi_{la}(rC_n)\leq r+2$.
\end{lemma}
\begin{proof}
We define a local antimagic labeling  $f:E\to \{1,2,\dots,nr\}$ as 
$$f(v_i^jv_{i+1}^j) =\begin{cases}
\frac{jn}{2}+ \frac{i}{2} +1&\mbox{if } i\equiv 0(\bmod\ 2), j\equiv 0(\bmod\ 2)\\

 (r- \frac{j}{2})n- \frac{i-1}{2} & \mbox{if } i\equiv 1(\bmod\ 2), j\equiv 0(\bmod\ 2)\\
 
 (r- \frac{(j-1)}{2})n- \frac{i}{2}- \frac{(n-1)}{2} &\mbox{if } i\equiv 0(\bmod\ 2),  j\equiv 1(\bmod\ 2)\\
 
 (\frac{j-1}{2})n+ \frac{i+1}{2} + \frac{n+1}{2} &\mbox{if } i\equiv 1(\bmod\ 2), j\equiv 1(\bmod\ 2).
 \end{cases}$$ 
 
Then the induced vertex weights are  as follows:

for $i\neq 0$
$$ w(v_i^j) =\begin{cases}
rn+2 &  \mbox{if } i\equiv j(\bmod\ 2) \\
rn+1 & \mbox{if } i\not\equiv j(\bmod\ 2).
\end{cases}$$

and for $i= 0$
$$ w(v_i^j) =\begin{cases}
nj+\frac{n+3}{2} & \mbox{if } j\equiv 0(\bmod\ 2).\\

2n(r-\frac{j-1}{2})-\frac{3}{2}(n-1) & \mbox{if } j\equiv 1(\bmod\ 2).
\end{cases}$$

Since we have $r+2$ distinct vertex weights, we conclude that $\chi_{la}(rC_n)\leq r+2$.
\end{proof}
From the Lemmas \ref{cycle1} and \ref{cycle2}, we have the following theorem.
\begin{theorem}
For $n \geq 3$ the local antimagic chromatic number $3 \le \chi_{la}(rC_n) \leq r+2$.
\end{theorem}
Let $rK_{1,n}$ denotes $r$ copies of a star $K_{1,n}$. Let $u_1,u_2,\dots,u_r$ be the $r$ central vertices of $rK_{1,n}$. Let $v_{i,1},v_{i,2},\dots,v_{i,n}$ be $n$ pendant vertices adjacent to the central vertex $u_i$.  Note that $deg(u_i)= n$ and $|E(rK_{1,n})| = rn$.
\begin{lemma}{\label{star1}}
For $r \geq 1$ and for even $n$ the $\chi_{la}(rK_{1,n}) = rn+1$. 
\end{lemma}
\begin{proof}
Define an edge labeling $f:E\to \{1,2,\dots,rn\} $ as:
\begin{equation}\nonumber
f(u_iv_{i,j}) = 
\begin{cases}
(i-1)\frac{n}{2}+j&\mbox{if } 1\leq j\leq \frac{n}{2}\\
rn-(n-j)-(i-1)\frac{n}{2}& \mbox{if } \frac{n}{2}+1\leq j\leq n.
\end{cases}
\end{equation}
Therefore the vertex  weights are,
$w(v_{i,j})= f(u_iv_{i,j}))$  and $w(u_i)= \frac{n}{2}(mn+1)$. Since pendant vertices contribute $rn$ distinct colors and each $u_i, 1\leq i\leq n$ has the same weight, the total number of distinct weights is $rn+1$.
\end{proof}
\begin{lemma}{\label{star2}}
For odd $r \geq 1$ and for odd $n$,  $\chi_{la}(rK_{1,n}) = rn+1$. 
\end{lemma}
\begin{proof}  
Since both $r$ and $n$ are odd, then by Theorem \ref{mr} there exists a magic rectangle $MR(n,r)$ of order $n\times r$. Let $C_1,C_2,\ldots,C_r$ be the $r$ columns of the magic rectangle $MR(n,r)$. Now we define a bijection $f:E\rightarrow \{1,2,\ldots,nr\} $ as $f(u_iv{i,j}) = C_{ij}$, where $C_{ij}$ is the $j$th entry of the $i^{th}$  column $C_i$. The vertex weights are, $w(u_i)= \frac{n(nr+1)}{2}$, $1\leq i\leq r$ and $w(v_{i,j}) = f(u_iv_{i,j})$. Since pendant vertices contribute $rn$ distinct colors and each $u_i, 1\leq i\leq n$ has the same weight,  therefore the total number of distinct weights are $nr+1$, hence $\chi_{la}(rK_{1,n}) = nr+1$. 
\end{proof}

\begin{lemma}{\label{star3}}
 For $r \geq 1$, $\chi_{la}(rK_{1,n}) = nr+2$, If
$r\equiv 0(\bmod ~2)$ and $n\equiv 1(\bmod ~2)$.
\end{lemma}
\begin{proof}
Since $n\equiv 1(\bmod ~2)$ and  $r \equiv 0(\bmod ~2)$, then $n\equiv r-1
\equiv 1(\bmod ~2)$, therefore by Theorem \ref{mr}, there exists $ n\times r-1$  magic rectangle $ MR(n, r-1)$. Let $ C_1, C_2,\ldots, C_{r-1}$ be the columns of the $ MR(n, r-1)$. We label the edges in the first $r-1$ copies of $ K_{1,n}$ using respective columns $ C_1, C_2,\ldots, C_{r-1}$. Label the edges in the $r^{th}$ copy of $ K_{1,n}$ using the remaining set of labels $\{n(r-1)+1,n(r-1)+2, \ldots,nr\}$. Now $ \sum_{i=1}^n n(r-1)+i = n^2(r-1)+\frac{n(n+1)}{2} $. The pendant vertices of $G$ induce $nr$ distinct weights. For the support vertices $\{u_1,u_2,\dots, u_r\}$, the weights are as follows:
$$w(u_i) = 
\begin{cases}
\frac{n(n(r-1))+1}{2}& \mbox{if } i = 1,2,\ldots, r-1\\
n^2(r-1) +\frac{n(n+1)}{2}& \mbox{if } i = r
\end{cases}$$

The total number of distinct weights under this labeling is $nr+2$, Hence we conclude that $\chi_{la}(rK_{1,n}) = rn+2$.
\end{proof}

The illustration of the local antimagic coloring of $3K_{1,3}$ is shown in Figure \ref{3K_{1,3}}.
\begin{figure}[ht]
    \centering
    \includegraphics[scale= .70]{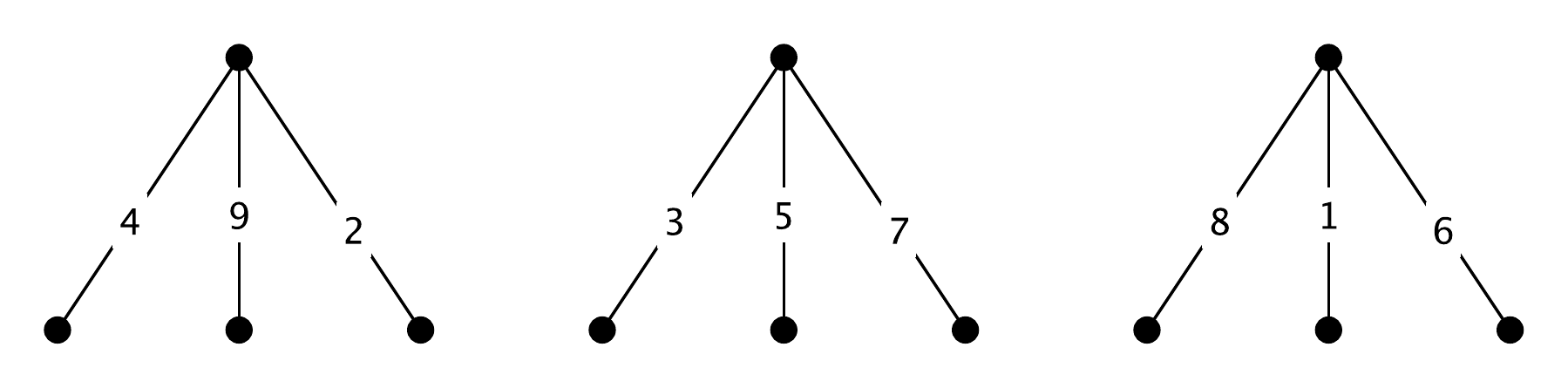}
    \caption{A local antimagic labeling of $3K_{1,3}$ with $\chi_{la}(3K_{1,3}) = 3\times 3+1$.}
    \label{3K_{1,3}}
\end{figure}

The following theorem is evident from the above Lemmas \ref{star1}, \ref{star2} and \ref{star3}.
\begin{theorem}
$rn+1 \le \chi_{la} (r K_{1,n}) \leq rn+2.$
\end{theorem}
The chromatic number of a complete graph on $n$ vertices is $n$. It is easy to observe and prove that $\chi_{la}(K_n) = n = \chi(K_n)$. Moreover, with some conditions on $n$, we have proved that $\chi_{la}(mK_n) = n$.
\begin{proposition} \label{prop:1}
For $n \ge 3, \chi_{la}(K_n) = n$.
\end{proposition}
\begin{proof}
Let $V(K_n) = \{ v_1, v_2, \dots, v_n\}$ and for each $1 \le j \le n$ and $j < i \le n$ let $e_{i-1} = v_jv_i$. Define edge labeling $f$ by $f(e_i) = (j-1)n - \frac{j(j-1)}{2}+1$. It is easy to observe that the weights, $w(v_1), \dots, w(v_n)$ are in increasing order. Hence, $\chi_{la}(K_n) \le n = \chi(K_n)$. This proves the proposition.
\end{proof}
\begin{proposition}
For $n \ge 3, n \equiv 1 \mbox{ or } 3(\bmod\ 4), \chi_{la}(mK_{n}) = n$.
\end{proposition}
\begin{proof}
Let $n \ge 3$. If $n \equiv 1 \pmod{4}$ then $K_n$ is $4r$-regular for some $r \ge 1$ and proof follows by Theorem~\ref{baca1}. If $n \equiv 3\pmod{4}$ then $K_n$ is $(4r+2)$-regular for some $r \ge 1$ and it contains $(n-1)$ even spanning cycles. Hence proof follows by Theorem~\ref{baca2}.\\
\end{proof}

\begin{theorem}
Let $\chi_{la}(rK_{m,n}) = 2$ if positive integers $m,\ n,\ r$ with $m \ne n$ satisfies one of the following conditions:
\begin{enumerate}
    \item $1 < m \le n$ and $m$ and $n$ are both even, $r \ge 1$, and $(m,n) \ne (2,2)$.
    \item $1 < m \le n$ and $m, n, r$ are all odd.
\end{enumerate}
\end{theorem}
\begin{proof}
By Theorem \ref{mrs2.3} on the existence of magic rectangle sets, we have the existence of MRS$(m,n;r)$ for each of the above cases. Suppose there is a magic rectangle set $MRS(m,n;r)$. Let $M_1,M_2,\dots, M_r$ denotes the $r$ magic rectangles in $MRS(m,n;r)$. For $1\leq k\leq r$, define the vertex set of $k^{th}$  copy of  $K_{m,n}$ as
$V =\{v_i^k, w_j^k : 1\leq i\leq n, 1\leq j\leq m \}$, where $\{v_i^k : 1 \le i \le m\}$ and $\{w_i^k : 1 \le i \le n\}$ form the respective partite sets of $k^{th}$ copy of $K_{m,n}$.\\

Now for each $k\ (1\leq k\leq r)$ and each $i\ (1 \le i \le n) $ we label the edge set $\{ v_i^k w_j^k : 1\leq j\leq m \}$ with the numbers in the $i$th row of $M_k$.
Since the sum of elements in any row or column in the magic rectangle set is equal, the resulting labelling is a local antimagic that induces $2$ different colors. Therefore, $\chi_{la}(rK_{m,n}) \le 2$. Also, $\chi(rK_{m,n}) = 2$. We conclude that $\chi_{la}(rK_{m,n}) = 2$ whenever there exists a MRS$(m,n;r)$.
\end{proof}
Figure \ref{fig:3k_2,4} illustrates the local antimagic labeling of $3K_{2,4}$.
\begin{figure}[ht]
    \centering
    \includegraphics[scale=.4]{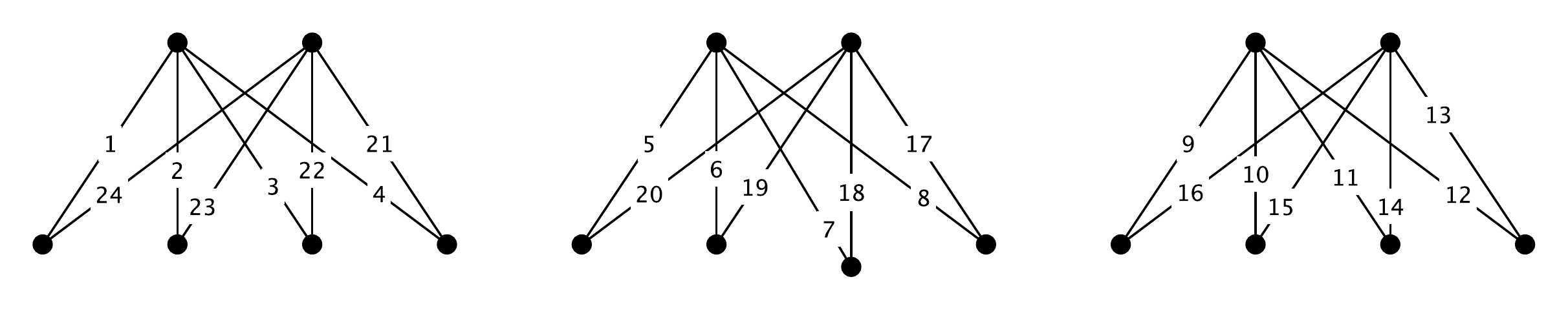}
    \caption{A local antimagic labeling of $3K_{2,4}$ with $\chi_{la}(3K_{2,4}) = 2$.}
    \label{fig:3k_2,4}
\end{figure}
\newline
An interesting class of graphs called a {\it necklace graph} has common vertices. A $u,v$-\textit{necklace} is a list of cycles $C_1, C_2, \dots, C_t$ such that 
$u \in C_1$, $v \in C_t$, consecutive cycles share exactly one vertex, and non-consecutive cycles are disjoint (see Figure \ref{fig:uv-necklace}). The number of edges in all the cycles is known as the \textit{length of the necklace}. We provide an upper bound for the local antimagic chromatic number for this class of graph.

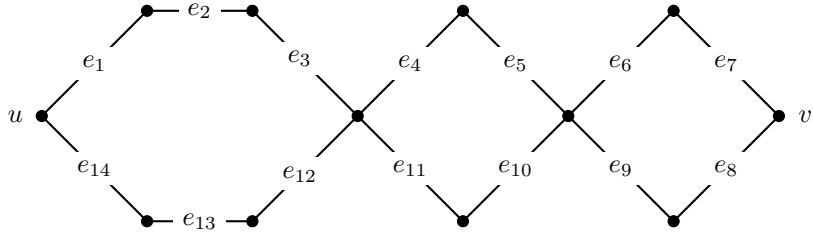
\begin{figure}[ht]
    \centering
\begin{tikzpicture}[scale=0.7]
\draw[fill=black] (0,0) circle (3pt);
\draw[fill=black] (2,2) circle (3pt);
\draw[fill=black] (4,0) circle (3pt);
\draw[fill=black] (2,-2) circle (3pt);
\draw[fill=black] (-2,-2) circle (3pt);
\draw[fill=black] (-4,-2) circle (3pt);
\draw[fill=black] (-4,2) circle (3pt);
\draw[fill=black] (-2,2) circle (3pt);
\draw[fill=black] (-6,0) circle (3pt);
\draw[fill=black] (6,2) circle (3pt);
\draw[fill=black] (8,0) circle (3pt);
\draw[fill=black] (6,-2) circle (3pt);

\draw [thick] (0,0) -- (2,2) -- (4,0) -- (2,-2) --(0,0) -- (-2,-2) --(-4,-2)-- (-6,0) -- (-4,2) -- (-2,2) --(0,0);
\draw[thick] (4,0) -- (6,2) -- (8,0) -- (6,-2) -- (4,0);
\node[fill=white] at (-5, 1) {$e_1$};
\node[fill = white] at (-3, 2) {$e_2$};
\node[fill = white] at (-1.1,1.1) {$e_3$};
\node[fill = white] at (1,1) {$e_4$};
\node[fill = white] at (3,1) {$e_5$};
\node[fill = white] at (5,1) {$e_6$};
\node[fill = white] at (7,1) {$e_7$};
\node[fill = white] at (7,-1) {$e_8$};
\node[fill = white] at (5,-1) {$e_9$};
\node[fill = white] at (3,-1) {$e_{10}$};
\node[fill = white] at (1,-1) {$e_{11}$};
\node[fill = white] at (-1.1, -1.1) {$e_{12}$};
\node[fill = white] at (-3,-2) {$e_{13}$};
\node[fill = white] at (-5,-1) {$e_{14}$};
\node at (-6.5,0) {$u$};
\node at (8.5,0) {$v$};
\end{tikzpicture}
    \caption{A $u,v$-necklace graph.}
    \label{fig:uv-necklace}
\end{figure}

\begin{theorem}
Let $G$ be an $u,v$-necklace on $t \ge 2$ cycles such that $G$ has no adjacent vertices of degree $4$. Then $\chi_{la}(G) \le 6$.

\end{theorem}
\begin{proof}
Let $G$ be an $u,v$-necklace of length $n$, where $C_i$ be a cycle of length $n_i$ for $1 \le i \le t$. By the definition of $G$, we have an Eulerian tour traversed clockwise starting and ending at $u$. We enumerate the edges of $G$ as we follow the Eulerian tour as shown in Figure \ref{fig:uv-necklace}. Now we label the edges as 
 \begin{equation*}
    f(e_i) =\begin{cases}
            \frac{i+1}{2} & \text{if $i$ is odd} \\
            n - (\frac{i}{2} - 1) & \text{if $i$ is even}.
            \end{cases}
\end{equation*}
Then for each vertex $x$ of degree $2$ other than $u$, it is easy to see that  $w(x) = n+1 \mbox{ or } n+2$ and
\begin{align*}
w(u) = 
\begin{cases}
    \frac{n+3}{2} \text{ if } n \text{ is odd}\\
    \frac{n+4}{2} \text{ if } n \text{ is even}
\end{cases}
\end{align*}
For a vertex $y$ of degree $4$, $w(y) = 2n+2 \text{ or } 2n+4 \text{ or } 2n+3$. This proves that $f$ induces $6$ colors as required.
\end{proof}
Let $G$ be a necklace graph with $q$ edges and lengths of all the cycles be even. Then $G$ is bipartite. In \cite{lac_cycle_lau} authors proved that for a bipartite graph $G$ with $q$ edges and two color classes $x$ and $y$, if number of vertices of colors $x$ and $y$ are $|X|$ and $|Y|$ respectively then $x|X| = y|Y| = \frac{q(q+1)}{2}$. Using this result one can obtain the examples when $\chi_{la}(G) > 2$. We have given one such example in Figure \ref{fig:necklace3c4}. Here, $q = 12$, $|X| = 6$, $|Y| = 4$ and $4 \nmid (\frac{q(q+1)}{2} = 6 \times 13)$. Hence, its local antimagic chromatic number is greater than $2$. We have the labeling which induces $3$ colors. Hence, its local antimagic chromatic number is $3$. Notice that if all the cycles in $G$ are of even length then $\chi_{la}(G) \ge 2$ and if at least one of the cycles is of odd length then $\chi_{la}(G) \ge 3$.
\begin{figure}[ht]
    \centering
\begin{tikzpicture}[scale=0.7]
\draw[fill=black] (0,0) circle (3pt);
\draw[fill=black] (2,2) circle (3pt);
\draw[fill=black] (4,0) circle (3pt);
\draw[fill=black] (2,-2) circle (3pt);
\draw[fill=black] (0,0) circle (3pt);
\draw[fill=black] (6,2) circle (3pt);
\draw[fill=black] (8,0) circle (3pt);
\draw[fill=black] (6,-2) circle (3pt);
\draw[fill=black] (10,2) circle (3pt);
\draw[fill=black] (12,0) circle (3pt);
\draw[fill=black] (10,-2) circle (3pt);

\draw [thick] (0,0) -- (2,2) -- (4,0) -- (6,2) --(8,0) -- (10,2) --(12,0)-- (10,-2) -- (8,0) -- (6,-2) -- (4,0) -- (2,-2) -- (0,0);
\node[fill=white] at (1, 1) {$2$};
\node[fill = white] at (3, 1) {$12$};
\node[fill = white] at (3, 1) {$12$};
\node[fill = white] at (5, 1) {$1$};
\node[fill = white] at (7, 1) {$7$};
\node[fill = white] at (9, 1) {$9$};
\node[fill = white] at (11, 1) {$5$};
\node[fill = white] at (11, -1) {$3$};
\node[fill = white] at (9, -1) {$11$};
\node[fill = white] at (7, -1) {$4$};
\node[fill = white] at (5, -1) {$10$};
\node[fill = white] at (3, -1) {$8$};
\node[fill = white] at (1, -1) {$6$};
\end{tikzpicture}
    \caption{A necklace graph $G$ with $\chi_{la}(G) =3$.}
    \label{fig:necklace3c4}
\end{figure}
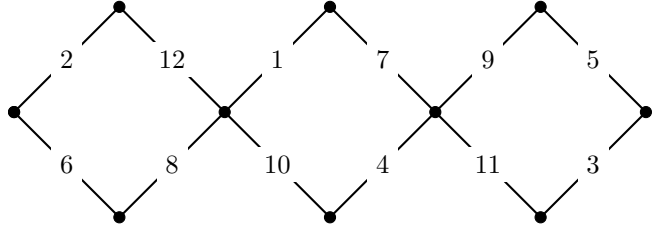
\section{Some Other Results}
Arumugam et al.\cite{lac_corona_arumugam} and Premalatha et al.\cite{lac_tree_premalatha} studied the local antimagic chromatic number of corona product $P_n \circ \overline{K}_m$, $C_n \circ \overline{K}_m$  and Setiawan et al.\cite{lac_corona_sugeng} have studied it for corona product $P_m \circ P_k$. The following theorem gives the bounds on the local antimagic chromatic number of corona product $G \circ \overline{K}_m$ for any graph $G$.

\begin{theorem}
Let $G$ be a graph with $p$ vertices and $q$ edges such that $\chi_{la}(G) = r$, if $m\equiv p(\bmod\ 2)$, 
then  $mp + 1 \leq \chi_{la}(G\circ \overline{K}_m)\leq mp + r$.
\end{theorem}

\begin{proof}
Since $\chi_{la}(G) = r$, there is a local antimagic bijection $f:E\to \{1,2,\dots,q \}$ with $r$ distinct weights. Further since $m\equiv p(\bmod\ 2)$, there exits a magic rectangle $MR(m,p)$ of order $m\times p$. Let $C_1, C_2,\dots,C_p$  be the $p$ columns of the magic rectangle $MR(m,p)$. Let $ u_1,u_2,\dots, u_p$ be the vertices and $e_1, e_2,\dots,e_q$ be the edges of the graph $G$. Let $v_i^{j}$ be the pendent vertices adjacent to the vertex $u_i$, $1\leq j\leq m$, $1\leq i\leq p$.\\

We define an edge labeling $g:G\circ \overline{K}_m\to \{1,2,\dots, q+mp \}$ by  $g(e_i)= f(e_i)$ and $g(u_iv_i^{j})= q + c_{ij}$, where $c_{i,j}$ is the $(i,j)$th entry of $MR(m,p)$.
Now weights of the vertices under $g$ are $w_g(u_i)= w_f(u_i)+ \frac{m(mp+1)}{2}+mq$, $w_g(v_i^{j}) = g(u_iv_i^{j})$. For a fixed $i$, $w_{g}(u_{i}) > g(v_{i}^{j}) = w_{g}(v_{i}^{j})$. Thus, we have $r+mp$ distinct weights.
  Hence $\chi_{la}(G\circ \overline{K_m})\leq r + mp$. Since there are $mp$ pendent vertices, by Theorem \ref{th:lac_bound}, $\chi_{la}(G\circ \overline{K_m})\geq mp + 1$. This proves the theorem.
\end{proof}
The bound in the Theorem 4.1 is not sharp e.g. $\chi_{la}(G \circ \overline{K_m})\geq mp + 2$ when $G$ is a path on $p$ vertices (Theorem 2.14, \cite{lac_corona_arumugam}). Also, $\chi_{la}(K_n \circ \overline{K_m}) = mn + n = |V(K_n \circ \overline{K_m})|$ for $m \ge 2, n \ge 3$. Therefore, the characterization of graphs $G$ on $p$ vertices for which $\chi_{la}(G \circ \overline{K_m}) = mp + k$, where $1 \le k \le mp$ is an open problem and it will appear in the subsequent papers.\\

We know that the order of a clique $G'$ of a graph $G$ is the lower bound of $\chi(G)$. A similar result holds for local antimagic chromatic number, as illustrated in the following lemma.
\begin{lemma} \label{clique}
If a graph $G$ contain a $k$-clique then $\chi_{la}(G)\geq k$.
\end{lemma}
\begin{proof}
Let $G^\prime$ be a k- clique in $G$ and let $f $ be a local antimagic labeling of $G$.
Since every vertex $v_i \in G^\prime$ is adjacent to $k-1$ other vertices and $f$ is local antimagic labeling of $G$, it follows that for every vertex pair $v_i, v_j \in G^\prime$, $ w(v_i)\neq w(v_j)$. Therefore the weights of the vertices of $G^\prime$ under $f$ are distinct, hence $\chi_{la}(G)\geq k$.
\end{proof}
\begin{lemma} \label{sub}
Let $G$ be a graph with vertex $v$ such that $deg(v) = \Delta(G) \ge 2$. Then, there is a subgraph $H$ of $G$ such that $\chi_{la}(H) = \Delta(G)+1$.
\end{lemma}
\begin{proof}
Let $G$ be a graph with vertex $v$ such that $deg(v) = \Delta(G)$. We consider a subgraph $H$ with vertex set $ \{v, v_i: vv_i \in E(G)\}$ and edge set $\{vv_i : vv_i \in E(G)\}$. Since $H$ is a star, $\chi_{la}(H) = \Delta(G)+1$.
\end{proof} 
We know that for a given subgraph $H$ of a graph $G$, $\chi(H) \le \chi(G)$. But this need not be the case in local antimagic chromatic numbers. Using Lemma  \ref{sub}, we give some explicit examples where the inequality does not hold. Lau et. al \cite{lac_lau} calculated the local antimgaic chromatic number of a bipartite graph and wheel:\\

\begin{align*}
\chi_{la}(K_{p,q}) &= 
\begin{cases} 
q+1 & \text{if } q>p=1\\
2 \quad & \text{if }  q > p \ge 2 \text{ and } p \equiv q(\bmod\ 2)\\
3 \quad &  \mbox{otherwise}\\
\end{cases}\\
\chi_{la}(W_{n}) &= 
\begin{cases} 
4 \quad & \text{if }  n \equiv 0 \text{ or } 1 \text{ or } 3(\bmod\ 4) \\
3 \quad & \mbox{otherwise}.
\end{cases}
\end{align*}

For $q > p \ge 2$ using construction given in Lemma~\ref{sub} with $G \cong K_{p,q}$ we obtain subgraph $H$ of $K_{p,q}$ with $\chi_{la}(H) = q + 1 > 3 \ge \chi_{la}(K_{p,q})$. Similarly for $G \cong W_n$, where $n \ge 5$, we obtain subgraph $H$ of $W_n$ such that $\chi_{la}(H) = n + 1 > 4 \ge \chi_{la}(W_n)$.\\

We pose the following problem.\\
\begin{problem}
Characterise graphs $G$ not containing $K_2$ components for which $\chi_{la}(H) \le \chi_{la}(G)$, for all connected subgraphs $H$ (not containing $K_2$ components) of $G$.
\end{problem}

The requirement for the subgraph to be connected is indispensable, as neglecting it could lead to straightforward counterexamples. Such counterexamples arise from graphs in the form of cycles of large lengths and subgraphs as the vertex-disjoint union of paths. Additionally, considering Lemma \ref{sub}, it's intuitive to consider graphs $G$ where the maximum degree is less than $\chi_{la}(G)$. However, this assumption doesn't hold true in all cases (see Example 2.5 in \cite{lac_cut_lau}).

\section{Construction}
In \cite{lac_arumugam} authors raised a question of characterising the graphs having the same chromatic and local antimagic chromatic number. Since, then a few examples are known where chromatic and local antimagic chromatic numbers are the same (Example 2.5 in \cite{lac_cut_lau}, Theorem 2.5 in \cite{lac_cycle2_lau}). The class of graphs having the same chromatic number and local antimagic chromatic number is rich. In this section, we give a recursive method to construct infinitely many graphs $\{G_i\}$ such that $\chi(G_i) = \chi_{la}(G_i)$ from the given graph $G$ with $\chi(G) = \chi_{la}(G)$. \\
\\
\noindent \textbf{Construction:}\\

Let $G$ be a local antimagic graph with local antimagic labeling $f_0$ such that $|E(G)| = m_0$ satisfying $\chi_{la} (G)= \chi(G)$. Let $|V(G)| = n \ge 4$ be even.\\

Let $q$ be a positive integer and consider $G_0 = G$. For each $i \ge 1$ consider $G_{i} = G_{i-1} + \overline{K}_q$, where $V(\overline{K}_q) = \{u_1, u_2, \dots, u_q\}$. Observe that $|E(G_i)| = m + \sum_{j=1}^{i} (n+(j-1)q)q = m_i$ (say) and that
\begin{equation} \label{eq:1}
\chi(G_i) = \chi(G_{i-1}) + 1.    
\end{equation}
First we show that $\chi_{la}(G_{i}) \le \chi_{la}(G_{i-1}) + 1$. Since we know that, if  $n \equiv q(\bmod\ 2)$, $n+(i-1)q \equiv q(\bmod\ 2)$ for each $i$, then there exists a magic rectangle $MR(n+(i-1)q, q)$. Therefore we assume that $q$ is even. Add $m_{i-1}$ to each entry of $MR(n+(i-1)q, q)$ to obtain a new magic rectangle $MR'$ of the same size in which row sum $(\rho)$, and column sum $(\sigma)$ are constant. Label the edges from $u_j$ to $V(G_{i-1})$ by $i$th column of $MR'$. Then $w(u_j) = \sigma$ for each $j$ and $w_{G_{i}}(x) = w_{G_{i-1}}(x) + \rho, \; \forall x \in V(G_{i-1})$. Since $q$ is even, we can choose $q$ so that $w_{G_{i}}(x) > w(u_{j})$ for all $i$ and $j$. This proves that $\chi_{la}(G_{i}) \le \chi_{la}(G_{i-1}) + 1$.\\

Now we show that $\chi(G_i) = \chi_{la}(G_i)$ for each $i$ using induction on $i$. For $i = 0$, the result is trivial. Suppose result is true for $i = t$ i.e. $\chi(G_t) = \chi_{la}(G_t)$. Then
\begin{align*}
\chi(G_{t+1}) &= \chi(G_{t} + \overline{K}_q)\\
              &= \chi(G_t)+1\\
              &= \chi_{la}(G_t)+1\\
              &\ge \chi_{la}(G_{t+1})
\end{align*}
and $\chi(G_{t+1}) = \chi(G_{t}) + 1 = \chi(G_{t} + \overline{K}_q) \le \chi_{la}(G_{t} + \overline{K}_q)$. This proves that $\chi(G_{t+1}) = \chi_{la}(G_{t+1})$. Hence by induction, the result is true for all $i \ge 0$. Thus, $\{G_i\}_i$ is the required sequence of graphs satisfying the property $\chi(G_i) = \chi_{la}(G_{i})$ for each $i$.\\

Now, we give one application of the above construction to calculate the local antimagic chromatic number of $r$-partite graphs in a particular case.
Let $t_1$ and $t_2$ be two integers such that $t_1 > t_2 \ge 2$ and $t_1 \equiv t_2(\bmod\ 2)$. Then $\chi_{la}(K_{t_{1}, t_{2}}) = 2$ (see, \cite{lac_lau}). Since $t_1 \equiv t_2(\bmod\ 2)$ and  $n = t_1 + t_2$ is even, therefore for each even $t_3$, we have $\chi_{la}(K_{t_1, t_2, t_3}) = 3$. Recursively applying above construction for the suitable choices of $t_i$s, we obtain $\chi_{la}(K_{t_1,t_2, \dots, t_r}) = r = \chi(K_{t_1,t_2, \dots, t_r})$.

\section{Conclusion and Scope} 
In this paper, we obtained the local antimagic chromatic number for unions of some graphs and a few others. There are infinitely many graphs $G$ for which $\chi_{la}(G) =  \chi(G)$, but a complete characterization is yet to be discovered.
\section{Acknowledgement}
This research is supported by the \textit{Science and Engineering Research Board} (Ref. No. CRG/2018/002536), Govt. of India.
\bibliographystyle{amsplain} 
\bibliography{refs} 
\end{document}